\documentclass[12pt,journal]{IEEEtran}
\onecolumn
\usepackage{mathrsfs}
\usepackage{blkarray}
\usepackage{amsmath}
\usepackage{amsthm}
\usepackage{graphicx}
\usepackage{amssymb}
\usepackage{enumerate}
\usepackage{longtable,tabularx,float}
\usepackage{cite}
\usepackage{mathcomp}
\usepackage{supertabular}
\usepackage{stmaryrd}
\usepackage{xcolor, color}
\usepackage{url}
\usepackage{algorithm}
\usepackage{algpseudocode}
\usepackage{bm}
\usepackage{rotating}
\usepackage{hyperref}
\usepackage[justification=centering, labelsep=period]{caption}

\newcommand{\sref}[2]{\hyperref[#2]{#1 \autoref*{#2}}}

\theoremstyle{plain}
\newtheorem{theorem}{Theorem}[section]

\newtheorem{lemma}{Lemma}[section]

\theoremstyle{definition}

\newtheorem{corollary}{Corollary}[section]

\newtheorem{construction}{Construction}[section]
\newtheorem{remark}{Remark}[section]

\newcommand{\B}{\mathcal {B}}

\newcommand{\C}{\mathcal {C}}

\newcommand{\bbZ}{{\mathbb Z}}
\newcommand{\bbF}{{\mathbb F}}

\begin{document}

\title{Improved Bounds for Codes over Trees}
\author{
Yanzhi~Li, \thanks{Y. Li ({\tt davidlee@mail.ustc.edu.cn}) and W. Zhong({\tt zhongwj@mail.ustc.edu.cn}) are with the School of Mathematical Sciences, University of Science and Technology of China, Hefei, 230026, Anhui, China.}
\and Wenjie~Zhong,
\and Tingting~Chen\thanks{T. Chen ({\tt ttchenxu@mail.ustc.edu.cn}) is with the Institute of Mathematics and Interdisciplinary Sciences, Xidian University, Xi'an 710071, China.}
\and and~Xiande~Zhang\thanks{X. Zhang ({\tt drzhangx@ustc.edu.cn})  is with the School of Mathematical Sciences, University of Science and Technology of China, Hefei, 230026, and with Hefei National Laboratory, University of Science and Technology of China, Hefei 230088, China. (Corresponding author)}}

\maketitle

\begin{abstract}
Codes over trees were introduced recently to bridge graph theory and coding theory with diverse applications in computer science and beyond. A central challenge lies in determining the maximum number of labelled trees over $n$ nodes with pairwise distance at least $d$, denoted by $A(n,d)$, where the distance between any two labelled trees is the minimum number of edit edge operations in order to transform one tree to another.   By various tools from graph theory and algebra, we show that when $n$ is large, $A(n,d)=O((Cn)^{n-d})$ for any $d\leq n-2$, and $A(n,d)=\Omega((cn)^{n-d})$ for any $d$ linear with $n$, where constants $c\in(0,1)$ and $C\in [1/2,1)$  depending on $d$. Previously, only $A(n,d)=O(n^{n-d-1})$ for fixed $d$ and $A(n,d)=\Omega(n^{n-2d})$ for $d\leq n/2$ were known, while the upper bound is improved for any $d$ and the lower bound is improved for $d\geq 2\sqrt{n}$. Further, for any fixed integer $k$, we prove the existence of codes of size $\Omega(n^k)$ when $n-d=o(n)$, and give explicit constructions of codes which show $A(n,n-4)=\Omega(n^2)$ and $A(n,n-13)=\Omega(n^3)$.

\end{abstract}


\section{Introduction}
Trees are fundamental structures in both theoretical and applied fields, offering efficient solutions for complex problems. In computer science, trees are used in data structures like binary search trees, AVL trees, and B-trees for fast data storage, retrieval, and indexing \cite{knuth1997art}. Decision trees are widely employed in artificial intelligence for classification and regression tasks \cite{breiman2017classification}.
In networking, tree structures optimize routing protocols and multicast communication, ensuring efficient data transmission across networks \cite{cormen2022introduction}, while in natural language processing (NLP), parse trees are essential for syntactic analysis \cite{jurafsky2009speech}. Furthermore, in cryptography, trees are applied to construct Merkle trees for data integrity and verification in blockchain technologies \cite{Mingli_Wu2019}.

In 2021, Yohananov and Yaakobi \cite{Yohananov2021} introduced the coding scheme over labelled trees, called \emph{codes over trees}, to correct  edge erasures. Here, the {\it tree-distance} between any two labelled trees is the minimum number of edit edge operations in order to transform one
tree to another.  There are several applications in which such codes can be used. For example, trees are widely used in computer data structures and are foundational in implementing various practical standard libraries \cite{Wiki2023,Microsoft2024,josuttis2012c++}, such as maps, sets, and other advanced data structures that optimize storage systems--one example being the Log-Structured Merge (LSM) tree \cite{o1996log,agrawal2008design,min2012sfs}. These tree-based data structures are typically organized hierarchically \cite{hopcroft1983data}, with each node storing a list of pointers to other nodes within the tree. However, in real-world systems, the unpredictability of cache storage and memory can lead to erroneous addresses \cite{ford2010availability,rashmi2013solution}, causing pointers to reference incorrect locations and thereby compromising the integrity and reliability of the data structure. To mitigate this issue, tree-based coding schemes can introduce redundant edges and nodes, which help to correct these unexpected pointer mismatches and ensure the system's reliability.

Motivated by these applications, it is worthwhile to explore  theoretical results about codes over trees under the tree-distance metric. The central problem regarding such codes is to 
determine the maximum cardinality of a code with a given vertex-set and minimum distance.
Denote $A(n,d)$ the maximum size of  codes over trees with $n$ nodes and tree-distance $d$.
By Cayley's formula \cite{cayley1889theorem}, it is known that $A(n,1)=n^{n-2}$. In \cite{Yohananov2021}, Yohananov and Yaakobi showed that $A(n,n-1)=\lfloor\frac{n}{2}\rfloor$ and $A(n,n-2)=n$,  and constructed optimal  codes by using Hamiltonian paths and spanning stars. For fixed distance $d$, they provided an upper bound $A(n,d)=O(n^{n-d-1})$;   for $d\leq \frac{n}{2}$, they gave a lower bound $A(n,d)=\Omega(n^{n-2d})$ by leveraging linear codes with Hamming distance; and they constructed codes of size $\Omega(n^2)$ when $d$ exceeds $\frac{n}{2}$.

In this paper, we improve the upper and lower bounds for $A(n,d)$ significantly and provide several explicit constructions when $n-d$ is some constant. Our main contributions are listed below.

\begin{enumerate}
    \item We present two new general upper bounds on $A(n,d)$, which combining extends the bound $A(n,d)=O(n^{n-d-1})$ from fixed $d$ to any $d\leq n-2$.
In fact our results show a stronger upper bound $A(n,d)=O((Cn)^{n-d})$ for some constant $C\in [1/2,1)$ depending on $d$ when $n$ is large. In particular, $A(n,d)\le 2\left(\frac{n}{2}\right)^{n-d-1}+O(n^{n-d-2})$ if $n-d$ is a constant.

    \item We give several general lower bounds
    which improve $A(n,d)= \Omega(n^{n-2d})$ asymptotically when $ d \geq 2\sqrt{n}$. Especially when $d=\delta n$ with $\delta\in (0,1)$, $A(n,d)=\Omega(cn)^{n-d}$ for constant $c\in (0,1)$ depending on $\delta$. Further, for any fixed integer $k\geq 1$, $A(n,d)= \Omega(n^k)$ when $d=n-\frac{(k+1)\ln n}{\ln\ln n}$.

         \item Four explicit constructions of codes over trees are presented, which imply $\Omega(n^2)=A(n,n-4)=O(n^3)$, $\Omega(n^3)=A(n,n-13)=O(n^{12})$, $ A(8,5)\geq 28$ and $A(11,8)\geq 35$. The first two constructions are by algebraic method,  the third one is by design theory, and the last two lower bounds support the guess that $A(n,n-3)=\frac{1}{2}n^2+\Theta(n)$.

\end{enumerate}

Our paper is organized as follows. In \autoref{pre}, we introduce necessary notations and preliminary results  of codes over trees. In \autoref{upperbound}, we show two upper bounds for $A(n,d)$ for general positive integer $d$. 
\autoref{con} gives  general lower bounds and four explicit constructions of codes.  Finally,
a brief conclusion is given in \autoref{conclusion}.

\section{Preliminaries}\label{pre}

Let $G=(V,E)$ be a simple graph with labelled vertices, where $V=[n]:=\{1,\ldots,n\}$  unless otherwise specified. A {\it spanning tree}  of $G$ is a connected acyclic subgraph of $G$ with $n-1$ edges. Two classes of  special spanning trees are  Hamiltonian paths and stars who have a vertex of degree $n-1$.
 The set of all spanning trees of the complete graph  $K_n$ over $[n]$ is denoted by $\mathcal{T}(n)$. A subset of $\mathcal{T}(n)$ is called a \emph{tree-code} with $n$ nodes,
and each tree in the code is called a \emph{tree-codeword}. The \emph{tree-distance} between two trees $T_1=([n], E_1)$ and $T_2=([n], E_2)$, denoted by $d(T_1,T_2)$, is defined to be
\[d(T_1,T_2)=n-1-|E_1\cap E_2|.\] It is known that the tree-distance is a metric \cite{Yohananov2021}. For a tree-code $\mathcal{C}\subseteq \mathcal{T}(n)$ with $n$ nodes, the tree-distance of $\mathcal{C}$, denoted by $d(\mathcal{C})$, is the minimum tree-distance between any two distinct tree-codewords in $\mathcal{C}$, that is, \[d(\mathcal{C})=\min_{T_1\neq T_2,T_1,T_2\in\mathcal{C}}\{d(T_1,T_2)\}.\]
  A tree-code $\mathcal{C}\subseteq \mathcal{T}(n)$ of tree-distance $d$ is called an $(n,|\mathcal{C}|,d)$ tree-code, which is able to correct any $d-1$ edge erasures, or any $\lfloor (d-1)/2\rfloor$ edge errors \cite{Yohananov2021}. For simplicity, we denote $\mathcal{C}$ as an $(n,d)$ tree-code if the size of the code is uncertain.

The largest size of an $(n,d)$ tree-code is denoted by $A(n,d)$. The fundamental problem is to determine the values of $A(n,d)$. 
Since $|\mathcal{T}(n)|=n^{n-2}$ by  Cayley's formula \cite{cayley1889theorem}, it is trivial that $A(n,1)=n^{n-2}$.
In \cite{Yohananov2021}, the authors determined  $A(n,n-1)=\lfloor\frac{n}{2}\rfloor$, which is achieved by a code consisting of $\lfloor\frac{n}{2}\rfloor$ edge-disjoint Hamiltonian paths \cite{lee2020decomposing};  and $A(n,n-2)=n$, which is achieved by all the $n$ different spanning stars. Further, noting that any path and any star share at most two edges, the collection of $\lfloor\frac{n}{2}\rfloor$ disjoint Hamiltonian paths and $n$ stars is an $(n, n-3)$ tree-code, which yields $A(n,n-3)\geq n+\lfloor\frac{n}{2}\rfloor$.

For general $d$, it is quite difficult to determine the exact values of $A(n,d)$. We summarize the known results obtained in \cite{Yohananov2021} as follows.

\begin{theorem}[\cite{Yohananov2021}]\label{thm_6_pre}
The followings hold.
\begin{itemize}
  \item[(1)] For all $n\geq 9$, $A(n,n-3)\leq n^2$.
  \item[(2)] For sufficiently large $n$ and fixed $d$, $A(n,d)=O(n^{n-d-1})$.
  \item[(3)] For any positive integer $d\leq n/2$, $A(n,d)=\Omega(n^{n-2d})$.
  \item[(4)] For fixed $m$ and prime $n$, there exists an $(n,\frac{n-1}{2}\cdot \lfloor \frac{n-1}{m}\rfloor,\lfloor \frac{3n}{4}\rfloor-\lceil \frac{3n}{2m}\rceil-2)$ tree-code.
\end{itemize}

\end{theorem}

The authors in \cite{Yohananov2021} expressed that although the general lower bound is $\Omega(n^{n-2d})$ for all $d\leq n/2$ in \autoref{thm_6_pre} (3), it is possible to construct a code of size $\Omega(n^2)$ with tree-distance $d$ approaching $ \frac{3n}{4}$ when $n$ is a prime number by \autoref{thm_6_pre} (4). It is interesting to find the range of $d$ for which $A(n,d)=\Theta(n^2)$.

The lower bound in \autoref{thm_6_pre} (3) is obtained by naturally viewing each tree with $n$ nodes as a binary vector of length $\binom{n}{2}$, indexed by all possible $\binom{n}{2}$ edges.
Then the tree-distance between two trees is half of the Hamming distance between the corresponding  binary vectors.  Hence, an $(n,d)$ tree-code is a subset of a binary code of length $\binom{n}{2}$ of Hamming distance $2d-1$. Then applying the pigeonhole principle to all cosets of such a linear code, they gave the lower bound in \autoref{thm_6_pre} (3). However, this lower bound is not good especially when the tree-distance $d$ is large. 

For the upper bounds, the authors in \cite{Yohananov2021} applied the sphere packing bound by considering the forest ball of a tree, that is, the set of all forests obtained by deleting a fixed number of edges from a tree. 
Note that deleting $d-1$ edges from a tree results in a forest with $d$ components (including isolated vertices). Denote by  $\mathcal{F}(n, d)$ the set of all forests over $n$ nodes with exactly $d$ components. The sphere packing bound in \cite{Yohananov2021} is
\begin{equation}\label{eq:sp}
  A(n,d)\leq |\mathcal{F}(n, d)|/ \binom{n-1}{d-1}.
\end{equation}
Combining with the value of $|\mathcal{F}(n, d)|$ in \cite{bollobas2012graph},
the authors obtained \autoref{thm_6_pre} (2) in \cite{Yohananov2021}.

The upper bound in \autoref{thm_6_pre} (1) was obtained by constructing an auxiliary bipartite graph with one part corresponding to the tree-code or its sub-code, and another part corresponding to the set of edges. This auxiliary graph will be of girth at least six. Then the authors applies Reiman’s inequality \cite{neuwirth2001size,reiman1958problem} to derive an upper bound on the code size. This idea will be refined to get a stronger upper bound in the next section.

\section{Upper bounds on the size of tree-codes}\label{upperbound}

In this section, we present two new general upper bounds on the size of tree-codes.
In \autoref{thm_6_pre} (2), we know that for any fixed  $d$, $A(n,d)=O(n^{n-d-1})$. Our first upper bound extends this result to any tree-distance $d\leq 0.99n$. 
Our second upper bound further extends this result to any $d\le n-2$,
which in fact shows a much stronger bound $A(n,d)=O((Cn)^{n-d})$ for some constant $C\in[1/2,1)$ depending on $d$ when $n$ is large. 
With a refinement of the second method, we give explicit new upper bounds for $d\in \{n-3,n-4, n-5\}$.

%

\subsection{The first upper bound}
The idea is to apply the sphere packing bound in Eq. (\ref{eq:sp})  with a more
 precise estimate of $|\mathcal{F}(n,d)|$ by double counting.

\begin{theorem}[\textbf{The first upper bound}]\label{firstb}
    For all $1\leq d\leq n-2$, $A(n,d)\le \frac{n^{n-d}}{n-d+1}$. In particular,  when $d=o(n)$ or $d= \delta n$ for some constant $\delta\in (0,1)$, we have $A(n,d)=O(n^{n-d-1})$.
\end{theorem}

\begin{proof}
Given a set $X\subseteq [n]$ with $d$ vertices, let $\mathcal{F}_{n,d}(X)\subseteq \mathcal{F}(n,d)$ be the set of all forests on $n$ vertices with $d$ components such that each vertex in $X$ is located in a different component. By  a generalized Cayley's formula in \cite{takacs1990cayley}, we know that
$|\mathcal{F}_{n,d}(X)|= dn^{n-1-d}:=F_{n,d}$ for any $X$ of size $d$ with $1\le d< n$. Define a set
\[I=\{(F,X):F\in \mathcal{F}(n,d),X\subseteq [n] \text{ with } |X|=d \text{ such that } F\in \mathcal{F}_{n,d}(X) \}.\]
We will compute the size of $I$ in two different ways.

First, there are $\binom{n}{d}$ ways to choose $X\subseteq [n]$ with $|X|=d$. For each such $X$, there are $|\mathcal{F}_{n,d}(X)|={F}_{n,d}$ forests $F\in \mathcal{F}(n,d)$ such that $F\in \mathcal{F}_{n,d}(X)$. Hence,
\[|I|=\binom{n}{d}\cdot {F}_{n,d}.\]

	On the other hand, there are $|\mathcal{F}(n,d)|$ ways to choose a forest $F\in \mathcal{F}(n,d)$.
	For each choice of $F$, there are at least $n-d+1$ ways to choose $X$ such that $F\in \mathcal{F}_{n,d}(X)$. In the extreme case,  one of the components of $F$ contains $n-d+1$ vertices, while each of the remaining components consists of exactly one isolated vertex. Hence,
	\[|I|\geq |\mathcal{F}(n,d)|\cdot (n-d+1).\]
	
Combining the above analysis, we see that
$|\mathcal{F}(n,d)|\cdot (n-d+1)\leq |I|=\binom{n}{d}\cdot F_{n,d}.$
Thus  \[|\mathcal{F}(n,d)|\le \frac{\binom{n}{d}\cdot {F}_{n,d}}{n-d+1} =\frac{\binom{n}{d}\cdot dn^{n-1-d}}{n-d+1}.\] By applying the sphere packing bound in Eq. (\ref{eq:sp}), we have

\[A(n,d)\le \frac{|\mathcal{F}(n,d)|}{\binom{n-1}{d-1}}\le \frac{\binom{n}{d}\cdot dn^{n-1-d}}{\binom{n-1}{d-1}(n-d+1)}= \frac{n^{n-d}}{n-d+1}.\]
\end{proof}

Note that when $n-d$ is sub-linear with $n$, say $d= n-n^\alpha$ for any $\alpha\in (0,1)$, we have $A(n,d)\le n^{n-d-\alpha}$ by \autoref{firstb},
which is weaker than $O(n^{n-d-1})$. Our second upper bound will give a much stronger bound for a wider range of $d$, which could imply a bound $O(n^{n-d-1})$ even when $n-d$ is a constant.

\subsection{The second upper bound}\label{secUB}
The idea is motivated by the proof techniques for Singleton bound in classical coding theory, which has been used in \cite{Yohananov2021} to obtain \autoref{thm_6_pre} (1). For convenience, we define the following notations.

For $1\leq  d \le n-2$ and $0\leq i\leq n-d-2$, an $(n,d;i)$ tree-code is an $(n,d)$ tree-code whose all tree-codewords intersect on a set of at least  $i$ common edges. Let $A(n,d;i)$ be the maximum size of all $(n,d;i)$ tree-codes. It is clear that $A(n,d;0)=A(n,d)$ and $A(n,d;i)\leq A(n,d)$ for $i\geq 1$. The following relation between $A(n,d;i)$ and $A(n,d)$ is obtained by the pigeonhole principle.

\begin{lemma}\label{m_nd2g}
For $i\geq 1$,	$A(n,d;i)\ge A(n,d;i-1)(n-i)/\binom{n}{2}$. Consequently,
\[A(n,d;n-d-2)\ge \frac{(d+2)(d+3)\cdots(n-1)}{\binom{n}{2}^{n-d-2}}A(n,d).\]
\end{lemma}
\begin{proof}
  Let $\mathcal{D}\subset T(n)$ be an $(n,d;i-1)$ tree-code of size $M=A(n,d;i-1)$. Suppose that all tree-codewords of $\mathcal{D}$ intersect on a set $E$ of  $i-1$ common edges. Then there are $n-i$ edges besides $E$ in each tree-codeword of $\mathcal{D}$.  By the pigeonhole principle, among all $\binom{n}{2}-(i-1)$ edges in $K_n$ excluding $E$, there exists an edge contained in at least $M(n-i)/(\binom{n}{2}-i+1)\geq M(n-i)/\binom{n}{2}$ trees of $\mathcal{D}$. Collecting all such trees, we get an $(n,d;i)$ tree-code, which completes the proof.
\end{proof}

To give an upper bound of $A(n,d;n-d-2)$, we construct an auxiliary bipartite graph as follows.  Let $\mathcal{D}\subset T(n)$ be an $(n,d;n-d-2)$ tree-code of size $A(n,d;n-d-2)$, whose all tree-codewords intersect on a set $E$ of  $n-d-2$ common edges. Let $\bar{E}$ be the set of all edges in $K_n$ excluding $E$. %
Then every tree-codeword in $\mathcal{D}$ contains exactly $d+1$ edges in $\bar{E}$, and every two tree-codewords in $\mathcal{D}$ share at most one common edge in $\bar{E}$. The desired bipartite graph has two parts $\mathcal{D}$ and $\bar{E}$, where an edge $e\in\bar{E}$ is connected to a tree $T\in \mathcal{D}$ if and only if $e$ is an edge of $T$. Denote this graph by $G\triangleq(\mathcal{D}\cup \bar{E}, \mathcal{E})$. Then it is easy to see that $G$ contains no cycles of length four and $|\bar{E}|=\binom{n}{2}-(n-d-2)$.

The following two upper bounds of $A(n,d;n-d-2)$ are obtained by simple double counting as proved, or by applying Reiman’s inequality \cite{neuwirth2001size,reiman1958problem} as used in  \cite{Yohananov2021}.

%



\begin{lemma}\label{m_nd2l}
\begin{itemize}
  \item[(1)] For any $1\leq d\le n-2$, $A(n,d;n-d-2)\le \binom{\binom{n}{2}}{2}/\binom{d+1}{2}$.
  \item[(2)] If $\binom{n}{2}-(n-d-2)<(d+1)^2$, $A(n,d;n-d-2) \le \frac{d\left(\binom{n}{2}-(n-d-2)\right)}{(d+1)^2-\left(\binom{n}{2}-(n-d-2)\right)}\leq  \frac{d\binom{n}{2}}{(d+1)^2-\binom{n}{2}}$.
\end{itemize}

\end{lemma}
\begin{proof} 
The first bound is obtained by double counting the number of $2$-stars centered in $\mathcal{D}$.
Let $S=\{(T,e_1,e_2): T\in \mathcal{D}, e_1,e_2\in \bar{E}, T \sim e_1, T\sim e_2\}$. When we fix $T\in \mathcal{D}$, we get $|S|=|\mathcal{D}|\binom{d+1}{2}$. However when we fix a pair $\{e_1,e_2\}\subset \bar{E}$, there exists at most one $T\in \mathcal{D}$ connected to both of them. So $|S|\le \binom{|\bar{E}|}{2}\le \binom{\binom{n}{2}}{2}$ and then $|\mathcal{D}|\le \binom{\binom{n}{2}}{2}/\binom{d+1}{2}$.

For the second one, it is equivalent to show that if $|\bar{E}|<(d+1)^2$, $|\mathcal{D}| \le \frac{d|\bar{E}|}{(d+1)^2-|\bar{E}|}$. By similarly double counting the number of 2-stars centered in $\bar{E}$, we  get
\begin{equation}\label{eqdstar}
                        \binom{|\mathcal{D}|}{2} \ge \sum_{e\in \bar{E}}\binom{d(e)}{2}.
                      \end{equation}
Here, $d(e)$ is the degree of $e$ in the graph $G$. By Jensen's Inequality, we have $\sum_{e\in \bar{E}}\binom{d(e)}{2}\geq |\bar{E}|\frac{\bar{d}(\bar{d}-1)}{2}$, where $\bar{d}$ is the average degree of vertices in $\bar{E}$. Inserting  $\bar{d}=\frac{(d+1)|\mathcal{D}|}{|\bar{E}|}$ into $\binom{|\mathcal{D}|}{2} \ge |\bar{E}|\frac{\bar{d}(\bar{d}-1)}{2}$, we get $|\mathcal{D}|((d+1)^2-|\bar{E}|)\leq d|\bar{E}|$. Then the result follows.
\end{proof}

Now we give our second upper bound of $A(n,d)$ for general $d$ by applying \sref{Lemmas}{m_nd2g} and \ref{m_nd2l}.

\begin{theorem}[\textbf{The second upper bound}]\label{upperbound_second}
Suppose that $1<d\leq n-2$ and $n$ is sufficiently large.
	\begin{enumerate}[(1)]
		\item When $d=\delta n$ with $\delta\in (0.1313,1)$, $A(n,d)\le 1.2(C_\delta n)^{n-d}$, where $C_\delta=\frac{e}{2}\delta^{\frac{\delta}{1-\delta}}\in (1/2,1)$.
		Note that $C_\delta$ is decreasing in $\delta$ and $C_\delta\to \frac{1}{2}$ as $\delta\to 1$.
		\item When $d=n-\theta$ with $2\leq \theta=o(n)$, we have
\begin{itemize}
   \item[(i)] $A(n,d)\le ((\frac{1}{2}+o(1))n)^{n-d-1}$ if $\Omega(\sqrt{n})=\theta=o(n)$;
		\item[(ii)] $A(n,d)\le 2.1\left(\frac{n}{2}\right)^{n-d-1}$ if $\theta=o(\sqrt{n})$; and
		\item[(iii)] $A(n,d)\le 2\left(\frac{n}{2}\right)^{n-d-1}+O(n^{n-d-2})$ if $\theta$ is a constant.
 \end{itemize}
		
		
\end{enumerate}

\end{theorem}

\begin{proof}
	When $d=\delta n$, we have $A(n,d;n-d-2)\le \frac{1}{4\delta^2}n^2$ by \sref{Lemma}{m_nd2l} (1). Then by \sref{Lemma}{m_nd2g} and Stirling's Formula,
	\begin{align*}
        A(n,d)&\le \frac{\frac{1}{4\delta^2}n^2\cdot \binom{n}{2}^{n-d-2}}{(d+2)\cdots(n-2)(n-1)}= \frac{1}{4\delta^2}n^2\left(\frac{n^2}{2}\right)^{n-d-2}\cdot \left(1-\frac{1}{n}\right)^{n-d-2}\cdot (d+1)n\frac{d!}{n!}\\
        &\le \frac{1+o(1)}{\delta e^{1-\delta}}\left(\frac{n^2}{2}\right)^{n-d}\cdot \frac{d!}{n!} \le
        \frac{1+o(1)}{\delta e^{1-\delta}}\left(\frac{n^2}{2}\right)^{n-d}\cdot \frac{\sqrt{2\pi d}(\frac{d}{e})^d}{\sqrt{2\pi n}\left(\frac{n}{e}\right)^n}\\
        &= \frac{1+o(1)}{\delta e^{1-\delta}}\left(\frac{n^2}{2}\right)^{n-d}\cdot \sqrt{\delta}(e\delta^{\frac{\delta}{1-\delta}}/n)^{n-d}\\
        &\le 1.2\left(\frac{e}{2}\delta^{\frac{\delta}{1-\delta}} n\right)^{n-d},
    \end{align*}  where the last inequality holds when $\delta\in (0.1313,1)$.

When $d=n-\theta$ with $2\leq \theta=o(n)$, 
we have $\binom{n}{2}\le (d+1)^2$ when $n$ is big. By  \sref{Lemma}{m_nd2l} (2), we have
	\[A(n,d;n-d-2)\le \frac{d\binom{n}{2}}{(d+1)^2-\binom{n}{2}}\le \frac{d\binom{n}{2}}{(n-\theta)^2-\frac{n^2}{2}}\le \frac{d\binom{n}{2}}{\frac{n^2}{2}-2n\theta}\le \frac{(n-\theta)n}{n-4\theta}\le n+4\theta\] for large $n$.
Similarly by \sref{Lemma}{m_nd2g} and $1-x\ge e^{-x(1+x)}$ for $0\le x\le \sqrt{2}-1$, we have
	\begin{align*}
		A(n,d)&\le \frac{(n+4\theta)\cdot \binom{n}{2}^{n-d-2}}{(d+2)\cdots(n-2)(n-1)}\\
		&\le \frac{(n+4\theta)\cdot \left(\frac{n^2}{2}\right)^{n-d-2}}{n^{n-d-2}\left(1-\frac{c-2}{n}\right)\cdots\left(1-\frac{2}{n}\right)\left(1-\frac{1}{n}\right)}\\
		&\le
		(2+8\theta/n)\left(\frac{n}{2}\right)^{n-d-1}\cdot e^{\sum_{i=1}^{\theta-2}\frac{i}{n}(1+\frac{i}{n})}\\
		&\le (2+8\theta/n)e^{\frac{\theta^2}{2n}\left(1+\frac{2\theta}{3n}\right)} \left(\frac{n}{2}\right)^{n-d-1}.
\end{align*}
When $\Omega(\sqrt{n})=\theta=o(n)$, we have $\theta/n=o(1)$, $e^{\frac{\theta^2}{2n}}=(1+o(1))^{n-d-1}$, and thus
\[A(n,d)\le \left((1/2+o(1))n\right)^{n-d-1}.\]
When $\theta=o(\sqrt{n})$, we have $\frac{\theta^2}{n}=o(1)$, $e^{\frac{\theta^2}{2n}}=1+o(1)$, and thus
\[A(n,d)\le (2+o(1))\left(\frac{n}{2}\right)^{n-d-1}\le 2.1\left(\frac{n}{2}\right)^{n-d-1}. \]
Especially when $\theta$ is a constant, we have  $e^{\frac{\theta^2}{2n}\left(1+\frac{2\theta}{3n}\right)}=e^{O(n^{-1})(1+O(n^{-1}))}=1+O(n^{-1})$, and thus
\[A(n,d)\le (2+O(n^{-1}))\left(\frac{n}{2}\right)^{n-d-1}= 2\left(\frac{n}{2}\right)^{n-d-1}+O(n^{n-d-2}). \]
\end{proof}

By \autoref{upperbound_second} (2)-(iii), we have
    $A(n,n-3)\le \frac{1}{2}n^2+O(n), A(n,n-4)\le \frac{1}{4}n^3+O(n^2)$ and $A(n,n-5)\le \frac{1}{8}n^4+O(n^3)$. In the next subsection, we present explicit upper bounds for these parameters.

\subsection{Explicit bounds for  $d=n-3,n-4,n-5$}

The following lemma is a refinement of \sref{Lemma}{m_nd2l} when $d=n-3,n-4,n-5$.

%
%
%
%
%

\begin{lemma}\label{n-345}
\begin{itemize}
  \item[(1)] For $n\geq 13$, $A(n,n-3;1)\le n+2$.
  \item[(2)] For $n\geq 35$, $A(n,n-4;2)\le n+4$.
    \item[(3)] For $n\geq 69$, $A(n,n-5;3)\le n+6$.
\end{itemize}
\end{lemma}
\begin{proof} For $n\geq 13$ and $d=n-3$,  we have $\binom{n}{2}-(n-d-2)<(d+1)^2$. By \sref{Lemma}{m_nd2l} (2), $|\mathcal{D}|=A(n,n-3;1)\le \frac{(n-3)\left(\binom{n}{2}-1\right)}{(n-2)^2-\left(\binom{n}{2}-1\right)}=\frac{(n-3)(n+1)}{n-5}\leq n+4$. Reconsider the graph $G=(\mathcal{D}\cup \bar{E}, \mathcal{E})$ defined before. In this case, $|\mathcal{E}|=|\mathcal{D}|(n-2)$ and $|\bar{E}|=\frac{(n-2)(n+1)}{2}$. So the average degree of each vertex in $\bar{E}$ is $\bar{d}= |\mathcal{E}|/|\bar{E}|= \frac{2(n-3)}{n-5}<3$.

If $\bar{d}\leq 2$, then $|\mathcal{D}|=|\mathcal{E}|/(n-2)\leq 2|\bar{E}|/(n-2)=n+1$.

If $2<\bar{d}<3$, then the right hand side of Eq. (\ref{eqdstar}) is minimized when all $d(e)$ is $2$ or $3$. Let $x,y$ be the number of $e\in \bar{E}$ whose degrees are $2$ and $3$, respectively. Then
	\begin{equation}
		\left\{
		\begin{array}{ll}
			x+y=\frac{(n-2)(n+1)}{2},\\
			2x+3y=|\mathcal{D}|(n-2),  \\
			x+3y\le \binom{|\mathcal{D}|}{2}.
		\end{array}
		\right.
	\end{equation}
Multiplying the second equation by $2$, subtracting the first one by $3$, and combining it with the third one, we get  $|\mathcal{D}|^2-(4n-7)|\mathcal{D}|+3(n-2)(n+1)\geq 0$. However, it is not possible when $|\mathcal{D}|=n+3$ or $n+4$.

For $d=n-4$ and $n-5$, the proofs are similar and thus omitted. \end{proof}

Combining \sref{Lemmas}{m_nd2g} and \autoref{n-345}, we have the following explicit upper bounds of $A(n,d)$ for $d=n-3,n-4,n-5$.

 \begin{theorem}\label{thm_iii.4}
     \begin{itemize}
        \item[(1)]  $A(n,n-3)\le \frac{1}{2}n(n+2)$ for $n\ge 13$;
		\item[(2)]  $A(n,n-4)\le \frac{1}{4} \left(n^3+5 n^2+6 n+12\right)$ for $n\ge 35$; and
		\item[(3)]  $A(n,n-5)\le \frac{1}{8} \left(n^4+9 n^3+28 n^2+92 n+292\right)$ for $n\ge 117$.
     \end{itemize}
 \end{theorem}
 Finally, we mention that the proof for the second upper bound uses the restricted intersection property of pairwise codewords, but does not use the tree property of each codeword. Further, as shown later in \autoref{smlcnst}, there exist $n-2$ spanning subgraphs (may contain cycles) each having $n-1$ edges, all of which intersect on a common edge and with pairwise intersection on at most $2$ edges. That is, without considering the tree property, \sref{Lemma}{n-345} (1) and thus \autoref{thm_iii.4} (1) cannot be improved much. In fact, we guess that the bound in \autoref{thm_iii.4} (1) is almost tight.

\section{Lower bounds on the size of tree-codes}\label{con}
 In this section, we  give several general lower bounds for $A(n,d)$. In \autoref{thm_6_pre} (3), we know that for  $d\leq n/2$, $A(n,d)=\Omega(n^{n-2d})$. Our results improve this lower bound for any $ d \ge 2\sqrt{n}$. Especially when $d=\delta n$ with $0<\delta<1$, we have $A(n,d)=\Omega(cn)^{n-d}$ for constant $c\in (0,1)$ depending on $\delta$. Further, we show that
 for any fixed integer $k\geq 1$, $A(n,d)= \Omega(n^k)$ when $d=n-\frac{(k+1)\ln n}{\ln\ln n}$.  Finally, we provide several explicit constructions of tree-codes, including $(n,\frac{(n-3)(n-9)}{9},n-4)$ tree-codes and $(n,\frac{n(n-2)^2}{32},n-13)$ tree-codes, which imply $\Omega(n^2)=A(n,n-4)=O(n^3)$ and $\Omega(n^3)=A(n,n-13)=O(n^{12})$. All these results show that even when $d$ approaches to $n$ with a constant gap, it is possible to have a code of size $\Omega(n^2)$ and even $\Omega(n^3)$.



\subsection{General lower bounds for $A(n,d)$}
The idea is to construct a large tree-code from a large independent set of a graph whose vertices are all spanning trees.



\begin{theorem}\label{lbged}
    For all $n>d\ge 1$, $A(n,d)\ge \binom{n-1}{d-1}^{-1}d^d n^{n-2d}$. Further if $d>\frac{n}{2}$,  $A(n,d)\ge \binom{n-1}{d-1}^{-1}(\frac{n}{2})^{n-d}$.
\end{theorem}
\begin{proof}
To show the first lower bound $A(n,d)\ge \frac{d^d}{\binom{n-1}{d-1}}n^{n-2d}$, we consider the graph $H$ with vertex set $\mathcal{T}(n)$, the set of all spanning trees over $[n]$, where two trees are connected if and only if they share at least $n-d$ common edges.  Note that $|\mathcal{T}(n)|=n^{n-2}$, and any independent set of $H$ is an $(n,d)$ tree-code.  Let $\alpha(H)$ be the independence number of $H$, then $A(n,d)\ge \alpha(H)$. Let $\Delta$ be the maximum degree of $H$, then it is known that $\alpha(H)\ge \frac{|\mathcal{T}(n)|}{\Delta+1}$ by \cite{bondy2008graph}. So it suffices to upper bound $\Delta$.

For any given spanning tree in $\mathcal{T}(n)$, there are $\binom{n-1}{n-d}$ ways to choose $n-d$ edges. For each choice, let $F$ be the forest with such $n-d$ edges and $d$ components. Let $q_1,q_2,\ldots,q_d$ be the component sizes of $F$, then $\sum_{i=1}^{d}q_i= n$. By a result from Lu, Mohr and S\'{z}ekely\cite{lu2013quest}, the number of spanning trees in $\mathcal{T}(n)$ that contain $F$ is
\[q_1q_2\cdots q_d n^{n-2-\sum_{i=1}^{d}(q_i-1)}= q_1q_2\cdots q_d n^{d-2}\le \left(\frac{\sum_{i=1}^{d}q_i}{d}\right)^dn^{d-2}= \left(\frac{n}{d}\right)^dn^{d-2}.\]
So $\Delta\le \binom{n-1}{n-d}(\frac{n}{d})^dn^{d-2}$, and
\[A(n,d)\ge \alpha(H)\ge \frac{|\mathcal{T}(n)|}{\Delta+1}\ge \frac{n^{n-2}}{\binom{n-1}{n-d}(\frac{n}{d})^dn^{d-2}}= \frac{d^d}{\binom{n-1}{d-1}}n^{n-2d}.\]
For $d>\frac{n}{2}$, since average component size $\sum_{i=1}^{d}q_i/ n\in(1,2)$ and $q_i$'s are integers, we have $\prod_{i=1}^n q_i\le 1^{2d-n}2^{n-d}= 2^{n-d}$. In this case, we have $\Delta\le \binom{n-1}{n-d}2^{n-d}n^{d-2}$, and
\[A(n,d)\ge \frac{|\mathcal{T}(n)|}{\Delta+1}\ge \frac{n^{n-2}}{\binom{n-1}{n-d}2^{n-d}n^{d-2}}= \frac{(\frac{n}{2})^{n-d}}{\binom{n-1}{d-1}}.\]
\end{proof}




One can easily verify that the equality $\frac{d^d}{\binom{n-1}{d-1}}n^{n-2d}= \frac{(\frac{n}{2})^{n-d}}{\binom{n-1}{d-1}}$ holds at $d=n/2$. So the first lower bound is better for $d\le n/2$ and the second is better for $d> n/2$.
We have the following corollary of \autoref{lbged} by considering different ranges of $d$.

\begin{corollary}\label{corlow} Suppose that $1<d\leq n-2$ and $n$ is sufficiently large.

	

	\begin{enumerate}[(1)]
		\item $A(n,d)\ge (\frac{d}{\sqrt{e}n})^{2d}n^{n-d}$ when $d=o(n)$.
		Note that for $d\ge 2\sqrt{n}$, we have $A(n,d)\geq(\frac{4}{e})^{d}n^{n-2d}$, improving the previous bound $\Omega(n^{n-2d})$; for $d= n^{1-o(1)}$, we have $A(n,d)\ge n^{n-d-o(d)}$.
		
		\item $A(n,d)\ge 2.5(c_\delta n)^{n-d+1/2}$ when $d=\delta n$ with $\delta\in (0,1)$, where
		$c_\delta=(1-\delta)\delta^{\frac{\delta}{1-\delta}}\max\{\delta^{\frac{\delta}{1-\delta}},\frac{1}{2}\}\in (0,1)$ is decreasing in $\delta$, and $c_\delta\to 1$ as $\delta\to 0$.
		
		\item $A(n,d)\ge (\frac{n-d}{2e})^{n-d}$ when $n-d=o(n)$. Moreover, given any fixed integer $k\ge 1$, we have $A(n,d)= \Omega(n^k)$ when $d=n-\frac{(k+1)\ln n}{\ln\ln n}$.
\end{enumerate}


\end{corollary}

\begin{proof}
We apply the two bounds $A(n,d)\ge \binom{n-1}{d-1}^{-1}d^dn^{n-2d}$ and $A(n,d)\ge \binom{n-1}{d-1}^{-1}(\frac{n}{2})^{n-d}$ from \autoref{lbged}.
	
	\begin{enumerate}
		
		\item When $d= o(n)$, since $\binom{n}{k}< (\frac{en}{k})^k$, we have
\[A(n,d)\ge \frac{d^d}{\binom{n-1}{d-1}}n^{n-2d}\ge \frac{n}{d}\frac{d^d}{(en/d)^d}n^{n-2d}\ge \left(\frac{d}{\sqrt{e}n}\right)^{2d}n^{n-d}. \]
		Moreover, for $d\ge 2\sqrt{n}$, $A(n,d)\ge\left(\frac{2}{\sqrt{en}}\right)^{2d}n^{n-d}=(\frac{4}{e})^{d}n^{n-2d}$. For $d= n^{1-o(1)}$, $A(n,d)\ge (\frac{d}{\sqrt{e}n})^{2d}n^{n-d}= (\frac{1}{\sqrt{e}n^{o(1)}})^{2d}n^{n-d}= n^{n-d-o(d)}$.

\item When $d=\delta n$ with $\delta\in(0,\frac{1}{2}]$, we have
\begin{align*}
	A(n,d)\ge& \frac{d^d}{\binom{n-1}{d-1}}n^{n-2d}= \frac{n}{d}\frac{d!(n-d)!}{n!}d^dn^{n-2d}\\
	\ge& (1-o(1))\frac{n}{d}\frac{\sqrt{2\pi d}(\frac{d}{e})^d\cdot\sqrt{2\pi(n-d)}(\frac{n-d}{e})^{n-d}}{\sqrt{2\pi n}(\frac{n}{e})^n}d^dn^{n-2d}\\
	\ge& (1-o(1))\sqrt{2\pi}(\frac{d}{n})^{2d-1/2}(n-d)^{n-d+1/2}\\
	\ge& 2.5((1-\delta)\delta^{\frac{2d-1/2}{n-d+1/2}}n)^{n-d+1/2}\\
	\ge& 2.5((1-\delta)\delta^{\frac{2\delta}{1-\delta}}n)^{n-d+1/2}.
\end{align*}
When $d=\delta n$ with $\delta\in(\frac{1}{2},1)$, similarly we have
\begin{align*}
	A(n,d)\ge&(1-o(1))\frac{n}{d}\frac{\sqrt{2\pi d}(\frac{d}{e})^d\cdot\sqrt{2\pi(n-d)}(\frac{n-d}{e})^{n-d}}{\sqrt{2\pi n}(\frac{n}{e})^n}\left(\frac{n}{2}\right)^{n-d}\\
	\ge&2.5\left(\frac{(1-\delta)\delta^{\frac{\delta}{1-\delta}}}{2}n\right)^{n-d+1/2}.
\end{align*}


Setting $c_\delta=(1-\delta)\delta^{\frac{\delta}{1-\delta}}\max\{\delta^{\frac{\delta}{1-\delta}},\frac{1}{2}\}$, we obtain the bound $A(n,d)\ge 2.5(c_\delta n)^{n-d+1/2}$ when $d=\delta n$ with $\delta\in(0,1)$.

		\item When $n-d=o(n)$, since $\binom{n}{k}< (\frac{en}{k})^k$, we have
        \begin{align*}
			A(n,d)\ge& \frac{1}{\binom{n-1}{d-1}}(\frac{n}{2})^{n-d}\ge  \frac{1}{\binom{n}{n-d}}(\frac{n}{2})^{n-d}\\
			\ge& \frac{1}{(\frac{en}{n-d})^{n-d}}(\frac{n}{2})^{n-d}\\
            =& \left(\frac{n-d}{2e}\right)^{n-d}.
		\end{align*}
Especially taking $d=n-\frac{(k+1)\ln n}{\ln\ln n}$, we have
        \begin{align*}
        A(n,d)&\ge (\frac{n-d}{2e})^{n-d}= \left(\frac{(k+1)\ln n}{2e\ln\ln n}\right)^{\frac{(k+1)\ln n}{\ln\ln n}}= n^{\frac{1}{\ln n}\cdot \frac{(k+1)\ln n}{\ln\ln n}\cdot \ln(\frac{(k+1)\ln n}{2e\ln\ln n})}= n^{(k+1)(1-o(1))}\\
        &\ge \Omega(n^k).
        \end{align*}
	\end{enumerate}
\end{proof}

\begin{remark}

When $d=\delta n$ with $\delta\in(0,1)$, we see a gap between the upper bound $A(n,d)=O((C_\delta n)^{n-d})$ in {\hyperref[firstb]{Theorems \ref*{firstb}}} and \ref{upperbound_second} with $1/2 <C_\delta\leq1 $, and the lower bound $A(n,d)= \Omega((c_\delta n)^{n-d})$ in Corollary \autoref{corlow} with $ 0<c_\delta< 1$. See \autoref{UB_LB_C}.  It is not hard to find that, as $\delta\rightarrow0$, the limit inferior of $C_\delta/c_\delta$ approaches 1, suggesting that the gap is minimized for $\delta\rightarrow0$.
Specifically, when $\delta=1/2$, we have $\Omega((\frac{n}{8})^{n/2})\leq A(n,\frac{n}{2})\leq O((\frac{en}{4})^{n/2})$.

\begin{figure}[!htbp]
	\centering
	\includegraphics[scale=0.4]{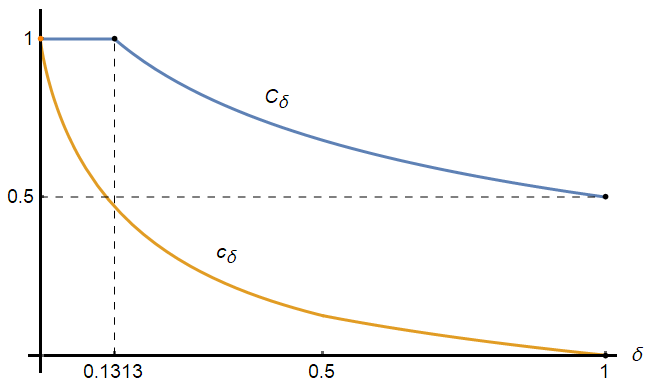}
	\caption{The values of $C_{\delta}$ and $c_{\delta}$ as functions of $\delta$}\label{UB_LB_C}
\end{figure}

	
	
\end{remark}


{}

\subsection{A construction of $(n,\frac{(n-3)(n-9)}{9},n-4)$ tree-codes}


Let $p$ be an odd prime, we will construct an $(n,\frac{(n-3)(n-9)}{9},n-4)$ tree-code  for $n=3p$, which consists of only Hamiltonian paths. 
 For convenience, a path with $n$ vertices is denoted by $i_1$-$i_2$-$\cdots$-$i_{n}$, where  $\{i_j,i_{j+1}\}$ is an edge for all $j\in [n-1]$.  

\begin{construction}\label{con_n-4}
Let $\bbF_p$ be a finite field of $p$ elements.
    Let $V_1,V_2$ and $ V_3$ be three disjoint sets such that $V=V_1\cup V_2\cup V_3$, where $V_i=\{(i,j):j\in \bbF_p\}$ for each $i=1,2,3$. Let $A= \{(a,b)\in (\mathbb{F}_p\backslash\{0\})^2:a\pm b\neq 0\}$, where $|A|=(p-1)(p-3)$. For each $(a,b)\in A$, clearly $a\mathbb{F}_p+b=\mathbb{F}_p$, hence $\{(i,aj+b):j\in \mathbb{F}_p\}$ is a permutation of $V_i$ for each $i=1,2,3$. For each $(a,b)\in A$, construct a Hamiltonian path $T_{a,b}$ over $V$ as
    \begin{align*}
T_{a,b}=& (1,0)\text{-}(2,a\cdot 0+b)\text{-}(3,b\cdot 0+a)\text{-}(1,1)\text{-}(2,a\cdot 1+b)\text{-}(3,b\cdot 1+a)\text{-}\cdots \\
&\text{-}(1,j)\text{-}(2,aj+b)\text{-}(3,bj+a)\text{-}\cdots\text{-}(1,p-1)\text{-}(2,a(p-1)+b)\text{-}(3,b(p-1)+a).
\end{align*}
Let  $\mathcal{C}_1=\{T_{a,b}:(a,b)\in A\}$.
\end{construction}

\begin{theorem}\label{thm_n-4}
    The code $\C_1$ in \sref{Construction}{con_n-4} is an $(n,\frac{(n-3)(n-9)}{9},n-4)$ tree-code.
\end{theorem}
\begin{proof}
When $(a,b),(a',b')$ are two distinct elements in $A$, clearly $T_{a,b}$ and $T_{a',b'}$ are also distinct. So $\mathcal{C}_1$ is a tree-code over $n$ nodes of size $(p-1)(p-3)=\frac{(n-3)(n-9)}{9}$.

Now we show that the tree-distance of $\mathcal{C}_1$ is $n-4$.
In other words, any two distinct paths $T_{a,b}$ and $T_{a',b'}$ share at most three common edges. Let
$E[V_i,V_j]$ denote the set of all edges between $V_i$ and $V_j$. Then
 by \sref{Construction}{con_n-4}, each edge of $T_{a,b}$ comes from one of $E[V_1,V_2],E[V_2,V_3]$ and $E[V_3,V_1]$. If $T_{a,b}$ and $T_{a',b'}$ share a common edge from $E[V_1,V_2]$, let the edge be $(1,x)\text{-}(2,ax+b)$ in $T_{a,b}$ and $(1,x)\text{-}(2,a'x+b')$ in $T_{a',b'}$. Then $ax+b=a'x+b'$. When $a=a'$, it follows that $b=b'$, which contradicts $(a,b)\neq (a',b')$. Thus, $a\neq a'$, and we have $x=\frac{b'-b}{a-a'}$, which is a unique solution in $\mathbb{F}_p$. Therefore, $T_{a,b}$ and $T_{a',b'}$ can share at most one edge from $E[V_1,V_2]$. Similarly, they  share at most one edge from $E[V_3,V_1]$.

It is left to show that $T_{a,b}$ and $T_{a',b'}$ share at most one edge from $E[V_2,V_3]$. Suppose they have a common edge  from $E[V_2,V_3]$, denoted as $(2,ax+b)\text{-}(3,bx+a)$ in $T_{a,b}$ and $(2,a'y+b')\text{-}(3,b'y+a')$ in $T_{a',b'}$. Then we obtain a linear system of equations:
\begin{equation}\label{eq-P2P3}
\left\{
  \begin{array}{ll}
    ax+b=a'y+b',  \\
    bx+a=b'y+a'.
  \end{array}
\right.
\end{equation}
When the coefficient matrix $\big(\begin{smallmatrix}
a & a' \\
 b & b'
\end{smallmatrix}\big)$
 is non-singular, the linear system (\ref{eq-P2P3}) has a unique solution, and we are done. When $\big(\begin{smallmatrix}
a & a' \\
 b & b'
\end{smallmatrix}\big)$
 is singular, we can set $a'=ka,b'=kb$ for some $k\in\mathbb{F}_p$ with $k\neq 0,1$ and turn the linear system (\ref{eq-P2P3}) into
\[\left\{
  \begin{array}{ll}
    a(x-ky)=(k-1)b,  \\
    b(x-ky)=(k-1)a.
  \end{array}
\right.\]
By eliminating $(x-ky)$, We obtain $(k-1)(a+b)(a-b)=0$. However, this is not possible since $k\neq 1$ and $a\pm b\neq 0$. This completes the proof.
\end{proof}

%
%

\subsection{A construction of $(n,\frac{n(n-2)^2}{32},n-13)$ tree-codes}

Let $q=3^m$ for some positive integer $m$, we will construct an $(n,\frac{n(n-2)^2}{32},n-13)$ tree-code for $n=2q$. 

\begin{construction}\label{con_n-13}
    Let $\bbF_q$ be a finite field of $q=3^m$  elements with a primitive element  $\omega$.  Let $V_1,V_2$ be two disjoint sets and $V=V_1\cup V_2$, where $V_i=\{(i,j):j\in \bbF_q\}$ for each $i=1,2$. Define
       \[B= \{(\omega^i,\omega^j,c): i=2,4,\ldots,q-1; j=1,3,\ldots,q-2; c\in \mathbb{F}_q\}.\]
%
For each $(a,b,c)\in B$, $\frac{b}{a}$ is not a square of $\mathbb{F}_q$.  According to \cite{dickson1896analytic}, the polynomial $x^3-\frac{b}{a}x$ is a permutation polynomial over $\mathbb{F}_q$, then $ax^3-bx+c=a(x^3-\frac{b}{a}x)+c$ is also a permutation polynomial over $\mathbb{F}_q$.
Then for each $(a,b,c)\in B$, we construct a spanning tree $T_{a,b,c}$ over $V$ 
with a stem path
\begin{align*}
	&(1,1)\text{-}(2,a\cdot 1^3-b\cdot 1+c)\text{-}(1,\omega)\text{-}(2,a\omega^3-b\omega+c)\text{-}\cdots\\
	&\text{-}(1,\omega^j)\text{-}(2,a\omega^{3j}-b\omega^j+c)\text{-}\cdots\text{-}(1,\omega^{q-2})\text{-}(2,a\omega^{3(q-2)}-b\omega^{q-2}+c)
\end{align*}
and a branch $(1,a)\text{-}(1,0)\text{-}(2,c)$. 

Let $\mathcal{C}_2=\{T_{a,b,c}:(a,b,c)\in B\}$.
\end{construction}

\begin{theorem}\label{thm_n-13}
    The code $\C_2$ in \sref{Construction}{con_n-13} is an $(n,\frac{n(n-2)^2}{32},n-13)$ tree-code.
\end{theorem}

\begin{proof} Note that $|B|=(\frac{q-1}{2})^2q=\frac{n(n-2)^2}{32}$. It  suffices to show that, for any two distinct $(a,b,c),(a',b',c')\in B$,  $T_{a,b,c}$ and $T_{a',b',c'}$ share at most twelve common edges.


%

By \sref{Construction}{con_n-13},  the tree $T_{a,b,c}$ consists of three types of edges:
\begin{center}
 $(1,x)\text{-}(2,ax^3-bx+c)$ for $x\neq 0$;~~~
 $(1,\omega x)\text{-}(2,ax^3-bx+c)$ for $x\neq 0,\omega^{q-2}$;~~~
 $(2,c)\text{-}(1,0)\text{-}(1,a)$.
\end{center}
Therefore, there are six cases for possible common edges of $T_{a,b,c}$ and $T_{a',b',c'}$.
\begin{description}
  \item[(1)] A common edge is $(1,x)\text{-}(2,ax^3-bx+c)$ in $T_{a,b,c}$ and $(1,x)\text{-}(2,a'x^3-b'x+c')$ in $T_{a',b',c'}$. Then $ax^3-bx+c=a'x^3-b'x+c'~(x\neq 0)$. When $a\neq a'$ and $c\neq c'$, it is a cubic congruence equation and we obtain at most three solutions, i.e., at most three common edges; otherwise, we easily obtain at most two common edges.
  \item[(2)] A common edge is $(1,\omega x)\text{-}(2,ax^3-bx+c)$ in $T_{a,b,c}$ and $(1,\omega x)\text{-}(2,a'x^3-b'x+c')$ in $T_{a',b',c'}$.
Then $ax^3-bx+c=a'x^3-b'x+c'~(x\neq 0,\omega^{q-2})$. The same as (1), when $a\neq a'$ and $c\neq c'$, it is a cubic congruence equation and we obtain at most three solutions, i.e., at most three common edges; otherwise, we easily obtain at most two common edges.
  \item[(3)] A common edge is $(1,x)\text{-}(2,ax^3-bx+c)$ in $T_{a,b,c}$ and $(1,\omega y)\text{-}(2,a'y^3-b'y+c')$ in $T_{a',b',c'}$.
Then $x=\omega y$ and thus $ax^3-bx+c=a\omega^3y^3-b\omega y+c=a'y^3-b'y+c'~(y\neq 0,\omega^{q-2})$, that is,
\[(a'-a\omega^3)y^3+ (b'-b\omega)y+ (c'-c)= 0~(y\neq 0,\omega^{q-2}).\]
Since $\frac{a'}{a}=\omega^{i_1-i_2}$ with $i_1,i_2$ even, we have $a'\neq a\omega^3$, we obtain a cubic congruence equation with at most three solutions, i.e., at most three common edges.
  \item[(4)] A common edge is $(1,\omega x)\text{-}(2,ax^3-bx+c)$ in $T_{a,b,c}$ and $(1,y)\text{-}(2,a'y^3-b'y+c')$ in $T_{a',b',c'}$.
The same as (3), we obtain a cubic congruence equation with at most three solutions, i.e., at most three common edges.
  \item[(5)] A common edge is $(1,0)\text{-}(1,a)$ in $T_{a,b,c}$ and $(1,0)\text{-}(1,a')$ in $T_{a',b',c'}$. The common edge exists when $a=a'$.
  \item[(6)] A common edge is $(1,0)\text{-}(2,c)$ in $T_{a,b,c}$ and $(1,0)\text{-}(2,c')$ in $T_{a',b',c'}$. The common edge exists when $c=c'$.
\end{description}
Combining the above cases, two distinct trees $T_{a,b,c}$ and $T_{a',b',c'}$ indeed have at most twelve common edges. This completes the proof.
\end{proof}

%

\subsection{Two small constructions for $d=n-3$}\label{smlcnst}
In \autoref{thm_iii.4} (1), we obtain that $A(n,n-3)\leq \frac{1}{2}n(n+2)$ for $n\geq 13$. In fact,  we guess that $A(n,n-3)\geq \frac{1}{2}n^2+cn$ for some constant $c$.  Here,  we give two constructions of tree-codes with $n=8$ and $n=11$ to support this guess.

\vspace{0.3cm}
\subsubsection{A construction of $(8,28,5)$ code}

The construction relies on balanced incomplete block designs (BIBDs).

    Let $v,k,\lambda$ be positive integers such that $v>k\geq 2$. A $(v,k,\lambda)$-BIBD is a pair $(X,\B)$, where $X$ is a finite set of points and $\B$ is a set of $k$-subsets of $X$, called \emph{blocks}, such that every pair of distinct points is contained in exactly $\lambda$ blocks. The following is a $(7,3,1)$-BIBD with $X=[7]$
and \begin{equation}\label{eqbibd}
      \mathcal{B}=\{\{4,5,6\}, \{1,6,7\}, \{1,2,4\}, \{2,5,7\}, \{2,3,6\}, \{3,4,7\}, \{1,3,5\}\}.
    \end{equation}


Next, we construct a table of size $8\times 8$ with each entry a block of size three
except the diagonal entries, which are empty, such that the seven blocks in the $i$th row (resp.\ the $j$th column) form a $(7,3,1)$-BIBD with point set $[8]\setminus \{i\}$ (resp.\ $[8]\setminus \{j\}$). See \autoref{tablei}. In fact, each row is obtained by reordering the blocks in (\ref{eqbibd}), and carefully mapping the points in $[7]$ to $[8]\setminus \{i\}$. For readers' convenience, in \autoref{tablep}, we list  all the mappings used to get \autoref{tablei}.  

\begin{table}[!ht]
	\centering
\caption{}\label{tablep}
	\begin{tabular}{|c|c|c|c|c|c|c|c|c|}
		\hline
		~ & 1 & 2 & 3 & 4 & 5 & 6 & 7 & 8 \\ \hline
		row & $\begin{matrix}
			1\rightarrow8\\
			6\rightarrow7\\
			7\rightarrow6
		\end{matrix}$
		&
		$\begin{matrix}
			2\rightarrow8\\
			5\rightarrow7\\
			7\rightarrow5
		\end{matrix}$&
		$\begin{matrix}
			3\rightarrow8\\
			6\rightarrow2\\
			2\rightarrow6
		\end{matrix}$ &
		$\begin{matrix}
			4\rightarrow8\\
			5\rightarrow6\\
			6\rightarrow5
		\end{matrix}$ &
		$\begin{matrix}
			5\rightarrow8\\
			4\rightarrow6\\
			6\rightarrow4
		\end{matrix}$ &
		$\begin{matrix}
			6\rightarrow8\\
			4\rightarrow5\\
			5\rightarrow4
		\end{matrix}$ &
		$\begin{matrix}
			7\rightarrow8\\
			1\rightarrow6\\
			6\rightarrow1
		\end{matrix}$ &
		id \\ \hline
		column &
		$\begin{matrix}
			1\rightarrow3\\
			3\rightarrow7\\
			7\rightarrow8
		\end{matrix}$ &
		$\begin{matrix}
			2\rightarrow3\\
			3\rightarrow5\\
			5\rightarrow8
		\end{matrix}$ &
		$\begin{matrix}
			3\rightarrow6\\
			5\rightarrow8\\
			6\rightarrow5
		\end{matrix}$
		&
		$\begin{matrix}
			1\rightarrow3\\
			3\rightarrow8\\
			4\rightarrow1
		\end{matrix}$  &
		$\begin{matrix}
			1\rightarrow8\\
			4\rightarrow1\\
			5\rightarrow4
		\end{matrix}$ &
		$\begin{matrix}
			1\rightarrow8\\
			5\rightarrow1\\
			6\rightarrow5
		\end{matrix}$
		&
		$\begin{matrix}
			2\rightarrow8\\
			6\rightarrow2\\
			7\rightarrow6
		\end{matrix}$ &
		$\begin{matrix}
			1\rightarrow6\\
			5\rightarrow1\\
			6\rightarrow7\\
			7\rightarrow5
		\end{matrix}$ \\ \hline
	\end{tabular}
\end{table}






\begin{table}[!htbp]
	\centering
	\caption{}\label{tablei}
	\begin{tabular}{|c|c|c|c|c|c|c|c|c|}
		\hline
		$N_{i,j}$ & 1 & 2 & 3 & 4 & 5 & 6 & 7 & 8  \\ \hline
		1 & / & \{4,5,7\} & \{2,5,6\} & \{3,5,8\} & \{6,7,8\} & \{2,4,8\} & \{3,4,6\} & \{2,3,7\}  \\ \hline
		2 & \{3,6,8\} & / & \{4,6,7\} & \{1,5,6\} & \{1,3,7\} & \{5,7,8\} & \{1,4,8\} & \{3,4,5\}  \\ \hline
		3 & \{4,7,8\} & \{1,5,8\} & / & \{2,6,8\} & \{1,4,6\} & \{1,2,7\} & \{2,4,5\} & \{5,6,7\}  \\ \hline
		4 & \{2,6,7\} & \{3,7,8\} & \{1,5,7\} & / & \{1,2,8\} & \{2,3,5\} & \{5,6,8\} & \{1,3,6\}  \\ \hline
		5 & \{2,3,4\} & \{4,6,8\} & \{2,7,8\} & \{3,6,7\} & / & \{1,3,8\} & \{1,2,6\} & \{1,4,7\}  \\ \hline
		6 & \{3,5,7\} & \{1,3,4\} & \{4,5,8\} & \{1,7,8\} & \{2,4,7\} & / & \{2,3,8\} & \{1,2,5\}  \\ \hline
		7 & \{2,5,8\} & \{3,5,6\} & \{1,6,8\} & \{1,2,3\} & \{3,4,8\} & \{1,4,5\} & / & \{2,4,6\}  \\ \hline
		8 & \{4,5,6\} & \{1,6,7\} & \{1,2,4\} & \{2,5,7\} & \{2,3,6\} & \{3,4,7\} & \{1,3,5\} & /  \\
		\hline
	\end{tabular}
	
\end{table}

Denote the block in the $(i,j)$th entry of \autoref{tablei} by $N_{i,j}$. Then it is easy to check that these $N_{i,j}$ with $i\neq j$ have the following  properties: 
	\begin{itemize}
			\item[(P1)] For each $i\neq j$,  $N_{i,j}\cap N_{j,i}=\emptyset$ and $N_{i,j}\cup N_{j,i}=[8]\backslash\{i,j\}$.
		\item[(P2)] 
In each row $i$, $|N_{i,j}\cap N_{i,j'}|=1$ for every $j\neq j'\in [8]\setminus\{i\}$.  It is the same for each column. 
		\item[(P3)] In each row $i$, for every $j\neq j'\in [8]\setminus\{i\}$,  $j'\in N_{i,j}$ if and only if $j\notin N_{i,j'}$. It is the same for each column.
	\end{itemize}

Based on \autoref{tablei}, we give our construction.

\begin{figure}[!htbp]
	\centering
	\includegraphics[scale=0.2]{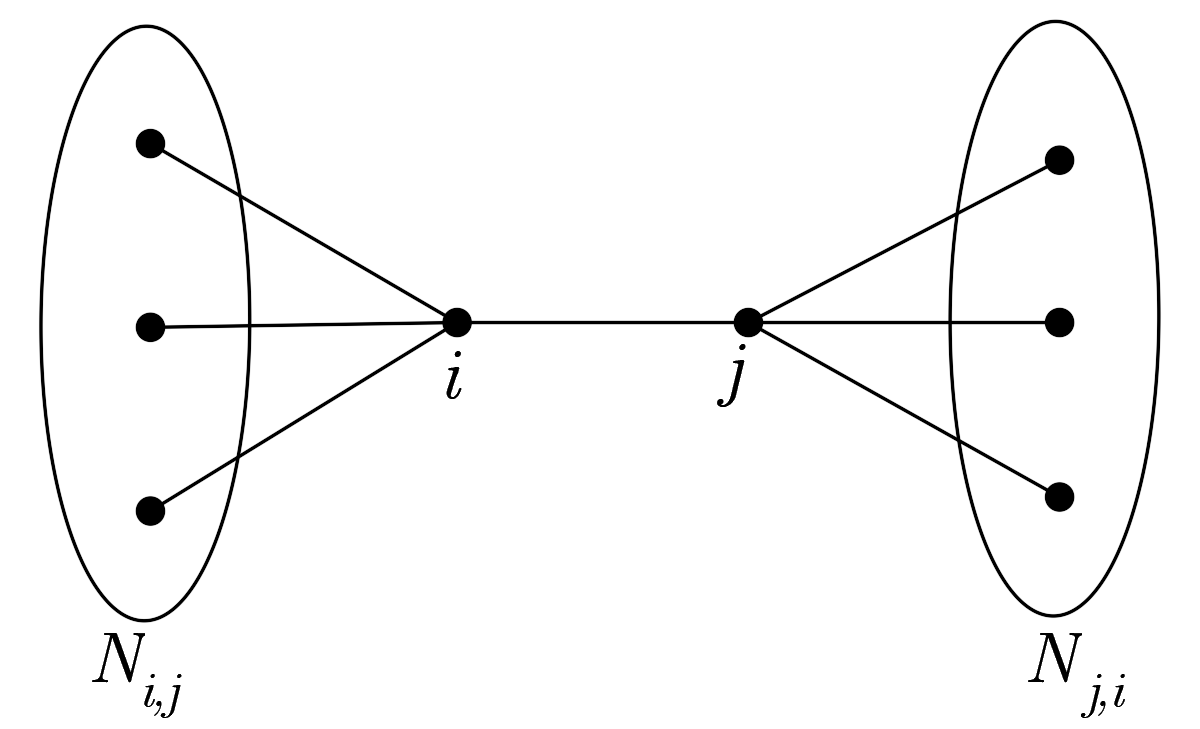}
	\caption{$S_{\{i,j\}}$: a star with two centers.}\label{figstar}
\end{figure}


\begin{construction}\label{con_n-14}For each pair $\{i,j\}\subset [8]$, let $S_{\{i,j\}}$ be the tree with vertices $i$ and $j$ connected,  each of which is further connected to vertices in $N_{i,j}$ and $N_{j,i}$, respectively. Let $\mathcal{C}_3=\{S_{\{i,j\}}:\{i,j\}\subset [8]\}$. The tree is well defined by (P1). See \autoref{figstar}, where each $S_{\{i,j\}}$ looks like a star with two centers.

\end{construction}



\begin{theorem}
	The code $\mathcal{C}_3$ in \sref{Construction}{con_n-14} is an $(8,28,5)$ tree-code.
\end{theorem}

\begin{proof} It is easy to see that for $\{i,j\}\neq \{i',j'\}$, $S_{\{i,j\}}\neq S_{\{i',j'\}}$. So $|\mathcal{C}_3|=\binom{8}{2}=28$. It is left to show that $S_{\{i,j\}}$ and $ S_{\{i',j'\}}$ share at most two edges.  We prove in two cases.
	\begin{enumerate}
		\item $|\{i,j\}\cap\{i',j'\}|=1$.
		Suppose $i=i'$ and $j\neq j'$. We split edges of each $S_{\{i,j\}}$ into two stars, the left star centered at $i$ with four leaves in $\{j\}\cup N_{i,j}$, and the right one centered at $j$ with three leaves in $N_{j,i}$. The two right stars of  $S_{\{i,j\}}$ and $S_{\{i,j'\}}$ have no common edge since $j'\in N_{j,i}$ if and only if $j\notin N_{j',i}$ by (P3). If the two stars are from different sides, they clearly have no common edge. It is left to check  the two left stars in $S_{\{i,j\}}$ and $S_{\{i,j'\}}$ both centered at $i$. By (P2), $|N_{i,j}\cap N_{i,j'}|=1$, and by (P3), $j'\in N_{i,j}$ if and only if $j\notin N_{i,j'}$. So $|(\{j\}\cup N_{ij})\cap (\{j'\}\cup N_{ij'})|\leq 2$. In this case, $S_{\{i,j\}}$ and $ S_{\{i',j'\}}$ share at most two edges.

%
		
		\item $|\{i,j\}\cap\{i',j'\}|=0$. Note that each edge in $S_{\{i,j\}}$ is incident to either $i$ or $j$.
However, there is exactly one edge incident to each $i$ and $j$ in $ S_{\{i',j'\}}$. 	So	$S_{\{i,j\}}$ and $ S_{\{i',j'\}}$ share at most two edges.
	\end{enumerate}
\end{proof}

Although we have successfully constructed an $(8,28,5)$ tree-code by utilizing the balanced properties of a $(7,3,1)$-BIBD, which has a very strong structure, it seems impossible to extend this method to large $n$ for a code with tree-distance $n-3$.

%

\vspace{0.3cm}

\subsubsection{A construction of $(11,35,8)$ tree-code} The construction is inspired by F\"{u}redi's algebraic construction \cite{furedi1996new}.
Given a prime number $p\ge5$, let $I_p$ be the set of representatives of $\{\{\omega,\omega^{-1}\}: \omega\in \bbZ_p\backslash\{0, 1,-1\}\}$. Then $|I_p|=(p-3)/2$.

%
%

For each $a\in \bbZ_p, b\in I_p$, define a graph $G_{a,b}$ with vertex set $\bbZ_p$ and edge set
\[E(G_{a,b})=\{\{x,y\}\subset \bbZ_p: y= bx+a\}.\]
For each $a\in [\frac{p-1}{2}]\subset \bbZ_p$, define $G_{a,1}$ over $\bbZ_p$ with
\[E(G_{a,1})=\{\{x,x+a\}:x\in \bbZ_p\setminus\{p-1\}\}.\]
Let $\mathcal{G}=\{G_{a,b}: a\in \bbZ_p, b\in I_p\}\cup \{G_{a,1}: a\in [\frac{p-1}{2}]\}$, then $|\mathcal{G}|=\frac{p^2-2p-1}{2}$. It can be verified that each $G_{a,b}\in \mathcal{G}$  has exactly  $p-1$ edges, and two different $G_{a,b}$'s share at most two common edges. However, each graph $G_{a,b}$ may not be a tree, since it may contain several disjoint cycles. Further,  for almost all edges $\{x,y\}\subset \bbZ_p$, there are $p-2$ subgraphs in $\mathcal{G}$ containing $\{x,y\}$, which is applicable to the final remark in \autoref{upperbound}.

For our purpose, we choose a permutation on $E(K_p)$ to map each $G_{a,b}$ to a new graph $G'_{a,b}$, which also has $p-1$ edges apparently, and hopefully the new graph is a spanning tree. Note that the permutation preserves the intersection size of two $G_{a,b}$'s. So we can collect all trees $G'_{a,b}$ after the permutation to get a tree-code of distance $p-3$.


Specially for $p=11$, by the assistance of computer,  we have found an edge permutation $\sigma$ of $E(K_{11})$ such that among the $49$ graphs $G'_{a,b}$, $35$ of them are trees. Thus we obtain an $(11,35,8)$ tree-code. See \autoref{tabii} for the list of $G'_{a,b}$'s which are trees. The permutation we used can be described as $\sigma=\sigma_1\sigma_2\sigma_3\sigma_4\sigma_5\sigma_6\sigma_7$, where $\sigma_1$ is a 10-cycle $\{2,3\}\rightarrow\{0,9\}\rightarrow\{2,8\}\rightarrow\{0,8\}\rightarrow\{3,6\}\rightarrow\{2,4\}\rightarrow\{1,8\}\rightarrow\{7,8\}\rightarrow\{0,6\}\rightarrow\{0,1\}\rightarrow\{2,3\}$, $\sigma_2$ and $\sigma_3$ are 3-cycles $\{4,6\}\rightarrow\{1,10\}\rightarrow\{8,10\}\rightarrow\{4,6\}$ and $\{4,9\}\rightarrow\{6,10\}\rightarrow\{5,10\}\rightarrow\{4,9\}$, $\sigma_4$ to $\sigma_7$ are transpositions $\{0,10\}\leftrightarrow\{4,8\},\{1,5\}\leftrightarrow\{2,9\},\{3,8\}\leftrightarrow\{4,5\},\{0,2\}\leftrightarrow\{0,3\}$.
\begin{table}[!ht]
	\centering
\caption{}\label{tabii}
	\begin{tabular}{|c|c|c|c|}
		\hline
		$\#$ & a & b & The edge set of $G'_{a,b}$ \\ \hline
		1 & 1 & 1 & \{\{2, 3\}, \{1, 2\}, \{0, 9\}, \{3, 4\}, \{3, 8\}, \{5, 6\}, \{6, 7\}, \{0, 6\}, \{8, 9\}, \{9, 10\}\} \\ \hline
		2 & 1 & 2 & \{\{0, 3\}, \{1, 3\}, \{1, 8\}, \{3, 5\}, \{1, 10\}, \{5, 7\}, \{6, 8\}, \{7, 9\}, \{4, 6\}, \{2, 8\}\} \\ \hline	
		3 & 1 & 3 & \{\{0, 2\}, \{1, 4\}, \{2, 5\}, \{2, 4\}, \{4, 7\}, \{5, 8\}, \{6, 9\}, \{7, 10\}, \{3, 6\}, \{1, 9\}\} \\ \hline
		4 & 1 & 4 & \{\{0, 4\}, \{2, 9\}, \{2, 6\}, \{3, 7\}, \{0, 10\}, \{5, 9\}, \{5, 10\}, \{0, 7\}, \{7, 8\}, \{1, 5\}\} \\ \hline	
		5 & 1 & 5 & \{\{0, 5\}, \{1, 6\}, \{2, 7\}, \{4, 5\}, \{6, 10\}, \{4, 9\}, \{0, 1\}, \{1, 7\}, \{0, 8\}, \{3, 9\}\} \\ \hline
		6 & 0 & 2 & \{\{1, 2\}, \{1, 8\}, \{2, 4\}, \{0, 10\}, \{4, 9\}, \{1, 6\}, \{3, 7\}, \{5, 8\}, \{7, 9\}, \{9, 10\}\} \\ \hline
		7 & 1 & 2 & \{\{2, 3\}, \{1, 3\}, \{2, 5\}, \{3, 7\}, \{6, 10\}, \{0, 5\}, \{2, 6\}, \{4, 7\}, \{6, 8\}, \{8, 9\}\} \\ \hline
		8 & 2 & 2 & \{\{0, 3\}, \{1, 4\}, \{2, 6\}, \{4, 5\}, \{4, 10\}, \{2, 9\}, \{2, 4\}, \{5, 7\}, \{0, 6\}, \{4, 8\}\} \\ \hline
		9 & 3 & 2 & \{\{0, 2\}, \{2, 9\}, \{2, 7\}, \{3, 9\}, \{0, 4\}, \{2, 5\}, \{1, 10\}, \{6, 7\}, \{9, 10\}, \{8, 10\}\} \\ \hline
		10 & 5 & 2 & \{\{0, 5\}, \{1, 7\}, \{1, 5\}, \{0, 2\}, \{1, 8\}, \{3, 8\}, \{0, 6\}, \{4, 6\}, \{1, 9\}, \{3, 10\}\} \\ \hline
		11 & 6 & 2 & \{\{0, 1\}, \{7, 8\}, \{2, 10\}, \{1, 3\}, \{3, 4\}, \{6, 7\}, \{7, 9\}, \{3, 6\}, \{1, 5\}, \{4, 10\}\} \\ \hline
		12 & 10 & 2 & \{\{4, 8\}, \{0, 9\}, \{3, 5\}, \{4, 7\}, \{5, 9\}, \{0, 1\}, \{2, 7\}, \{0, 10\}, \{6, 9\}, \{4, 6\}\} \\ \hline
		13 & 0 & 3 & \{\{1, 3\}, \{2, 6\}, \{3, 9\}, \{1, 4\}, \{3, 8\}, \{6, 7\}, \{7, 10\}, \{0, 8\}, \{5, 9\}, \{4, 6\}\} \\ \hline
		14 & 1 & 3 & \{\{2, 3\}, \{1, 4\}, \{2, 7\}, \{3, 10\}, \{1, 8\}, \{6, 8\}, \{0, 7\}, \{4, 5\}, \{6, 9\}, \{9, 10\}\} \\ \hline
		15 & 2 & 3 & \{\{0, 3\}, \{2, 9\}, \{0, 8\}, \{0, 2\}, \{3, 4\}, \{5, 6\}, \{6, 9\}, \{1, 7\}, \{0, 10\}, \{7, 9\}\} \\ \hline
		16 & 3 & 3 & \{\{0, 2\}, \{1, 6\}, \{1, 5\}, \{1, 3\}, \{5, 7\}, \{5, 10\}, \{2, 7\}, \{5, 8\}, \{8, 9\}, \{4, 8\}\} \\ \hline
		17 & 4 & 3 & \{\{0, 4\}, \{1, 7\}, \{2, 10\}, \{0, 9\}, \{3, 8\}, \{5, 8\}, \{0, 1\}, \{3, 7\}, \{6, 8\}, \{8, 10\}\} \\ \hline
		18 & 6 & 3 & \{\{0, 1\}, \{1, 9\}, \{1, 2\}, \{3, 4\}, \{4, 7\}, \{4, 9\}, \{2, 6\}, \{5, 7\}, \{2, 8\}, \{3, 10\}\} \\ \hline
		19 & 7 & 3 & \{\{0, 7\}, \{8, 10\}, \{3, 5\}, \{0, 10\}, \{0, 5\}, \{2, 4\}, \{6, 7\}, \{8, 9\}, \{1, 9\}, \{4, 10\}\} \\ \hline
		20 & 9 & 3 & \{\{2, 8\}, \{1, 8\}, \{3, 7\}, \{4, 10\}, \{2, 5\}, \{5, 6\}, \{0, 6\}, \{3, 6\}, \{3, 9\}, \{5, 10\}\} \\ \hline
		21 & 10 & 3 & \{\{4, 8\}, \{1, 2\}, \{2, 5\}, \{4, 5\}, \{0, 4\}, \{3, 5\}, \{7, 9\}, \{7, 8\}, \{6, 10\}, \{7, 10\}\} \\ \hline
		22 & 0 & 5 & \{\{2, 9\}, \{2, 10\}, \{3, 4\}, \{6, 10\}, \{3, 5\}, \{6, 8\}, \{2, 7\}, \{0, 6\}, \{1, 9\}, \{5, 10\}\} \\ \hline
		23 & 1 & 5 & \{\{2, 3\}, \{1, 6\}, \{0, 3\}, \{3, 5\}, \{4, 10\}, \{3, 8\}, \{6, 9\}, \{3, 7\}, \{1, 5\}, \{7, 10\}\} \\ \hline
		24 & 3 & 5 & \{\{0, 2\}, \{7, 8\}, \{3, 7\}, \{1, 4\}, \{5, 6\}, \{0, 1\}, \{5, 7\}, \{4, 6\}, \{6, 10\}, \{9, 10\}\} \\ \hline
		25 & 5 & 5 & \{\{0, 5\}, \{8, 10\}, \{1, 8\}, \{3, 9\}, \{3, 4\}, \{5, 8\}, \{2, 6\}, \{7, 8\}, \{6, 9\}, \{4, 8\}\} \\ \hline
		26 & 6 & 5 & \{\{0, 1\}, \{2, 3\}, \{2, 5\}, \{3, 10\}, \{5, 9\}, \{2, 4\}, \{0, 6\}, \{0, 8\}, \{7, 9\}, \{8, 10\}\} \\ \hline
		27 & 7 & 5 & \{\{0, 7\}, \{2, 6\}, \{0, 2\}, \{3, 8\}, \{4, 9\}, \{1, 10\}, \{7, 9\}, \{4, 5\}, \{8, 9\}, \{2, 10\}\} \\ \hline
		28 & 0 & 7 & \{\{1, 7\}, \{0, 9\}, \{3, 10\}, \{1, 10\}, \{2, 5\}, \{6, 9\}, \{5, 7\}, \{7, 8\}, \{8, 9\}, \{4, 10\}\} \\ \hline
		29 & 1 & 7 & \{\{2, 3\}, \{7, 8\}, \{1, 8\}, \{0, 2\}, \{4, 7\}, \{3, 5\}, \{5, 10\}, \{6, 7\}, \{0, 8\}, \{4, 9\}\} \\ \hline
		30 & 4 & 7 & \{\{0, 4\}, \{2, 3\}, \{2, 7\}, \{4, 10\}, \{5, 6\}, \{2, 6\}, \{7, 9\}, \{5, 8\}, \{1, 9\}, \{4, 6\}\} \\ \hline
		31 & 5 & 7 & \{\{0, 5\}, \{0, 8\}, \{3, 4\}, \{0, 4\}, \{5, 7\}, \{2, 4\}, \{7, 10\}, \{6, 8\}, \{1, 5\}, \{9, 10\}\} \\ \hline
		32 & 6 & 7 & \{\{0, 1\}, \{1, 2\}, \{1, 5\}, \{3, 5\}, \{1, 4\}, \{5, 8\}, \{1, 10\}, \{0, 7\}, \{0, 6\}, \{3, 9\}\} \\ \hline
		33 & 7 & 7 & \{\{0, 7\}, \{1, 3\}, \{2, 10\}, \{2, 4\}, \{1, 8\}, \{5, 9\}, \{5, 6\}, \{1, 7\}, \{6, 10\}, \{4, 8\}\} \\ \hline
		34 & 8 & 7 & \{\{3, 6\}, \{1, 4\}, \{0, 3\}, \{3, 7\}, \{3, 4\}, \{4, 9\}, \{2, 7\}, \{8, 9\}, \{5, 9\}, \{8, 10\}\} \\ \hline
		35 & 9 & 7 & \{\{2, 8\}, \{2, 9\}, \{1, 2\}, \{4, 5\}, \{0, 5\}, \{6, 7\}, \{3, 7\}, \{4, 6\}, \{6, 9\}, \{2, 10\}\} \\ \hline
	\end{tabular}
\end{table}

\section{Conclusion}\label{conclusion}
In this paper, we present improved upper and lower bounds on the maximum size of tree-codes, especially $\Omega((c_\delta n)^{n-d})=A(n,d)=O((C_\delta n)^{n-d})$, where $d=\delta n$ with $\delta\in(0,1)$, and constants $c_\delta\in (0,1)$ and $C_\delta\in (1/2,1]$. We also present several explicit constructions of codes with $n-d$ is some constant.  However, asymptotically determining the maximum size of codes for general $d$ is a bit challenging, even for determining the range of $d$ for which $A(n,d)=\Theta(n^2)$.



\section*{Declarations}

\textbf{{Conflict of interest}} The authors have no conflicts of interest to declare that are relevant to the content of this paper.

 \textbf{Funding} This work is supported by the National Key Research and
		Development Programs of China 2020YFA0713100 and 2023YFA1010200, the NSFC
		under Grants No. 12171452 and No. 12231014, and the Innovation Program for Quantum Science and Technology 2021ZD0302902.

 \textbf{Data availability} This article does not have any associated data.



\end{document}